\documentclass[12pt]{article}
\usepackage{amssymb,amsmath, amsthm}
\usepackage{color,xcolor}
\usepackage[normalem]{ulem}
\newcommand{\eb}{\begin{equation}}
\newcommand{\ee}{\end{equation}}
\newcommand{\ebx}{\begin{equation*}}
\newcommand{\eex}{\end{equation*}}

\newtheorem{lemma}{Lemma}[section]
\newtheorem{proposition}[lemma]{Proposition}
\newtheorem{theorem}[lemma]{Theorem}

\newtheorem{corollary}[lemma]{Corollary}

\newcommand{\delete}[1]{{\color{red}\ifmmode\text{\sout{\ensuremath{#1}}}\else\sout{#1}\fi}}

\allowdisplaybreaks

\makeatletter
\renewcommand*\env@matrix[1][*\c@MaxMatrixCols c]{%
  \hskip -\arraycolsep
  \let\@ifnextchar\new@ifnextchar
  \array{#1}}
\makeatother

\begin{document}

	\title{Some results related to Macaulay's Theorem about Hilbert functions and applications}
	
\author{Yun Gao\footnote{School of Mathematical Sciences, Shanghai Jiao Tong University, Shanghai, People's Republic of China. \textbf{Email:}~gaoyunmath@sjtu.edu.cn.
Supported
by NFSC, No.12471042 and No.12471078}}

	\maketitle

\begin{abstract}

Let  $I$  be a homogeneous ideal in the polynomial ring  $R = k[z_1, \cdots, z_n]$ , where  $k$  is an algebraically closed field of characteristic zero. Macaulay's Theorem provides constraints on the Hilbert function of  $I$  or $R/I$ from one degree to the next.

Nowadays, the standard quotation of Macaulay's theorem is $H_{R/I}(d + 1) \le H_{R/I}(d)^{\langle d\rangle}$, which is regarding the quotient $R/I$ and the combinatorial computation in the formula involves the number $d$ explicitly. However, the origin statement of Macaulay is in fact regarding the Hilbert function of $I$ itself and the relevant combinatorics explicitly involves the number of variables (i.e. $n$) and does not depend on $d$.

In this paper, we provide an elementary proof of the equivalence between these two versions of Macaulay's theorem.
The original degree-independent version is more suitable for problems such as those involving sums of polynomial squared norms.

Motivated by the Hermitian analogue of Hilbert's 17th problem and proper holomorphic mappings between complex unit balls,
some questions lead to the study of Hermitian polynomials  $M(z, \bar{z}) \in \mathbb{C}[z_1, \ldots, z_n, \bar{z}_1, \ldots, \bar{z}_n]$  satisfying  $M(z, \bar{z})\|z\|^{2l} = \|h\|^2$  for some  $l$  and a holomorphic mapping  $h = (h_1, \cdots, h_R)$ . 
Using Macaulay's Theorem, we derive new inequalities relating  $n$,  $l$, the signature  $(p, q)$  of the coefficient matrix of  $M(z, \bar{z})$ , and  $R$  (the rank of  $M(z, \bar{z})\|z\|^{2l}$ ) and extend these results to norms of
arbitrary signatures, which hold uniformly for all bidegrees of $M(z, \bar{z})$.

\end{abstract}
 
 {\bf Key words:} Hilbert function, graded ideal, sum of squared norms, proper mapping
 
\section{Introduction}

In the study of Hilbert functions of standard graded algebras, Macaulay's theorem \cite{Ma} holds a foundational position. 
Specifically, Macaulay's theorem characterizes the maximal possible growth of Hilbert functions from one degree to the next.

Let us first of all define a few notations. Every positive integer $A$ can be written as certain sum of binomial coefficients as follows. For every $n\in\mathbb N^+$, there exist unique positive integers $a_n>a_{n-1}>\cdots>a_\delta > 0$, where $\delta\geq 1$ and $a_j\geq j$ for every $j$, such that
$A=\binom{a_n}{n}+\cdots+\binom{a_\delta}{\delta}$. This is called the {\it $n$-th Macaulay's representation of $A$}. These representations naturally appeared in the works of Macaulay~\cite{Ma} and Green~\cite{Gr} on homogeneous ideals in polynomial rings and are used very often later in various works related Macaulay's theorem.
Let $A_{(n)}=\binom{a_n}{n}+\cdots+\binom{a_\delta}{\delta}$ be the $n$-th Macaulay's representation of $A$. For any integers $s$ and $t$, define

$$
A_{(n)}|_s^t=\binom{a_n+t}{n+s}+\cdots+\binom{a_\delta+t}{\delta+s}. 
$$

Let $k$ be a field of characteristic zero and $R=k[z_1, . . . , z_n]$ be the polynomial ring. An ideal $I\subset R$ is homogeneous if it can be generated by a set of homogeneous polynomials. Let $I \subset R$ be a homogeneous ideal. Then both $I$ and $R/I$ are graded modules over the graded ring $R$.  The Hilbert function of $I$ is defined by $H_I(d):=\dim_k I_d$, where $I_d$ is the degree-$d$ component of $I$. Similarly, the Hilbert function of $R/I$, denoted $H_{R/I}$, is
the function where $H_{R/I}(d) :=\dim_k(R/I)_d = \dim_k R_d - \dim_k I_d =\binom{n-1+d}
{d} - H_I(d)$ for $d >0$.

In all the literature within the last few decades that we have looked up, Macaulay's theorem has been always quoted as a statement regarding $H_{R/I}$, as follows (see \cite{AMS}, \cite{BGM},\cite{CM}, \cite{Gr}, etc.) using our notations.

\begin{theorem}\label{Mac}
	 (Macaulay's Theorem on $H_{R/I}$ \cite{Gr}) Let $I \subset R = k[z_1, . . . , z_n]$ be a homogeneous ideal. Then,	 
$$	
H_{R/I}(d + 1) \le H_{R/I}(d)_{(d)}|_{1}^{1}
$$
for $d \ge 0$.
\end{theorem}

Thus, the Macaulay's theorem above regulates the possible growth of Hilbert function of $R/I$ from one degree to the next. Moreover, the combinatorial computation at each degree $d$ involves explicitly the integer $d$ as the $d$-th Macaulay 's representation of $H_{R/I}(d)$ is used. 

Motivated by the problem of proper holomorphic mappings in Several Complex Variables, the author and Ng proved in~\cite{GN3} a combinatorial identity related to the Macaulay's representations and the operators $A_{(n)}|^t_s$ defined above. From this identity the author successfully converted the Macaulay's theorem above on $H_{R/I}$ to a statement regarding the Hilbert function $H_I$ for $I$ itself (see Theorem~\ref{M2} below), and the combinatorial computation at each degree only involves explicitly the integer $n$ (the number of variables), but not $d$ (the degree). Initially, we believed we had discovered a new
formulation of Macaulay's theorem.

However, after looking up Macaulay's original article~\cite{Ma}, we found that our ``new'' version of Macaulay's theorem is in fact his {\it original} statement in~\cite{Ma}. Actually, in his paper Macaulay did not write his result in a usual theorem form but one can still read from page 537 of~\cite{Ma} that his original statement is just Theorem~\ref{M2} below. Being puzzled by this discrepancy between Macaulay's original statement of his theorem and its modern quotations by mathematicians (i.e. Theorem~\ref{Mac}), we had  tried quite hard to dig up the literature to see how his theorem changed to the modern version. However, we are unable to find any reference in which both versions of the theorem appear. Around the period 1920s-1930s, the original version was still cited by mathematicians in the literature, e.g. in ~\cite{Sp, Wh}. On the other hand, since the late 1970s, the theorem started to be quoted as a version similar to Theorem~\ref{Mac}. For the period from the 1930s to the early 1970s, there seems to be no obtainable literature related to Macaulay's theorem. Then, the combinatorialists were among the first group of mathematicians citing and studying the generalizations of Macaulay's theorem after the aforementioned blank period of citations. They proved various combinatorial generalizations of Macaulay's theorem (see~\cite{CL,Cl}). In addition, they linked the combinatorial problems of finite sets, simplicial complexes, etc. to the study of Hilbert functions of standard graded algebras~\cite{St1, St2}. We tend to believe that the combinatorialists were the ones who changed the original Macaulay's theorem on $H_I$ to a theorem on $H_{R/I}$. Most notably, we found that in~\cite{St3}, R. Stanley wrote ``{\it The explicit numerical form in which we state the result first appeared in~\cite{St1} and is also considered in~\cite{St2}"}. 

As mentioned, there seems to be no available literature explicitly relating the two versions of Macaulay's theorem. We deem it useful to give the detail of this linkage in this article. Moreover, due to the fact that the combinatorial computation involved in the original version of Macaulay's theorem only explicitly depends on the number of variables but not the degree, we found its applications in tackling certain algebraic problems, such as the problem of representing polynomials as sums of squared norms (see Section \ref{GSOS}).

For a Hermitian polynomial  $M(z,\bar{z}) \in \mathbb{C}[z_1,\ldots,z_n,\bar{z}_1,\ldots,\bar{z}_n]$  (i.e., a polynomial that takes real values), the question of whether  $M(z,\bar{z})\|z\|^2$  can be expressed as a sum of squared norms of holomorphic polynomials has attracted considerable interest. This problem is partly motivated by the Hermitian analogue of Hilbert’s 17th problem, the classical sums of squares (SOS) problem, and also frequently arises in the study of proper holomorphic mappings between complex unit balls. For further context on these connections, we refer readers to \cite{Da}, \cite{CD1}, \cite{CD2}, \cite{GLV}, and related works.

A foundational result, independently established by Quillen \cite{Qu} and later by Catlin and D'Angelo \cite{CD2}, asserts that if  $M(z,\bar{z})$  is strictly positive on the unit sphere  $S^{2n-1}$ , then there exists a natural number  $l$  such that  $M(z,\bar{z})\|z\|^{2l}$  admits a representation as a sum of squared norms:  
\begin{eqnarray}\label{S}
	M(z,\bar{z})\|z\|^{2l}=\sum_{i=1}^{R}|h_i(z)|^2=\|h(z)\|^2,
\end{eqnarray}  
where the  $h_i(z)$  are linearly independent holomorphic polynomials on  $\mathbb{C}^n$ .

This result naturally raises two core questions: (1) What conditions on  $M(z,\bar{z})$  guarantee such a representation? (2) What is the minimal  $l$  for which  $M(z,\bar{z})\|z\|^{2l}$  can be written as a sum of squared norms? Prior works have addressed these questions: \cite{DZ} and \cite{TY} derived effective bounds on  $l$  under general additional assumptions, while \cite{BG2} provided algebraic conditions (in terms of Betti numbers) for the existence of such  $l$ .

Paralleling interest in Hilbert's 17th problem, mathematicians have also focused on the rank  $R$  of the representation in Equation (\ref{S}), defined as the dimension of the vector space spanned by the  $h_i(z)$ . A key result involving this rank, known as “Huang's Lemma,” finds widespread use in CR geometry (see, e.g., \cite{HJY2}, \cite{EHZ}, \cite{Xiao} and references therein).

More recently, Ebenfelt \cite{Eb} posed the SOS conjecture regarding the possible rank of  $M(z,\bar{z})\|z\|^2$  as  $M(z,\bar{z})$  varies. In \cite{HL} and \cite{GK}, the authors derived conditions linking the signature  $(p, q)$  of  $M(z,\bar{z})$, the dimension  $n$, and the rank  $R$  of  $M(z,\bar{z})\|z\|^{2l}$  to the representability of  $M(z,\bar{z})\|z\|^{2l}$  as a sum of squared norms. Only some special cases of this conjecture have been solved (see the relevant sections in those works).

 Building on SOS conjecture, the authors of \cite{GN2} generalized the conjecture to arbitrary signatures: specifically, they considered representations of the form 
$$M(z,\bar{z})\|z\|_{s,t}^2 = |h_1|^2 + \cdots + |h_{s'}|^2 - |h_{s'+1}|^2 - \cdots - |h_{s'+t'}|^2,$$
which correspond to proper mappings  $h = [h_1, \cdots, h_{s'+t'}]$  between generalized balls. Crucially, classifying such “multipliers”  $M(z,\bar{z})$  enables the classification of the associated proper mappings (see \cite{GN1}, \cite{DL}, etc.). Thus, investigating the relationships between the rank of  $M(z,\bar{z})$, the dimension  $n$  of  $\mathbb{C}^n$, and the rank  $R$  of  $M(z,\bar{z})\|z\|_{s,t}^{2}$  emerges as a natural direction.

These prior works have directly informed our research. In this paper, we establish some new inequalities relating  $n$,  $l$,  $p$,  $q$, and  $R$  and extend the results to arbitrary signatures by leveraging the degree-independent formulation of Macaulay’s theorem. All inequalities hold uniformly for all bidegrees of $M(z, \bar{z})$

\section{
The oringal Macaulay's theorem on $H_I$
}

We now state a key combinatorial lemma relating the operations
defined above, which thus connects Macaulay's  theorem on $H_{R/I}$ and the original Macaulay's Theorem on $H_I$, i.e. Theorem~\ref{M2}.
To streamline the presentation, we defer its proof to the Section \ref{App}.

\begin{lemma}\label{binom}
	Suppose integers $m,d,s\geq 1$ and $A,B\geq 0$. If $A+B=\binom{m+d}{d}$, then $$A_{(m)}|_0^{s}+B_{(d)}|^{s}_{s}=\binom{m+d+s}{d+s}.$$
\end{lemma}

We can now state Macaulay's original theorem on $H_I$ using our notations:

\begin{theorem}\label{M2}(\cite{Ma}P537)
	Let $I \subset R = k[z_1, . . . , z_n]$ be a homogeneous ideal. Then
	$$H_I(d+1)\ge H_I(d)_{(n-1)}|_{0}^{1}$$
\end{theorem}

\begin{proof}
For $d>0$, $H_{R/I}(d) +H_I(d)=\dim_k(R/I)_d +\dim_k I_d= \dim_k R_d  =\binom{n-1+d}{d}$.  By Lemma \ref{binom} 
$$H_{I}(d)_{(n)}|^{1}_{0}+H_{R/I}(d)_{(d)}|^{1}_{1}=\binom{n+d}{d+1}.$$

Theorem \ref{Mac} implies $H_{R/I}(d + 1) \le H_{R/I}(d)_{(d)}|_{1}^{1}$. Hence $H_I(d+1)\ge H_I(d)_{(n-1)}|_{0}^{1}$

\end{proof}

A simple calculation gives the following corollary which is weaker than Theorem \ref{M2}.

\begin{corollary}
Let $I \subset R = k[z_1, . . . , x_n]$ be a homogeneous ideal. Then
$$H_I(d)\le H_I(d+1)_{(n-1)}|_{0}^{-1}$$

\end{corollary}

\section{Application to Generalized SOS problem}\label{GSOS}

Denote $\|z\|$ as the Euclidean norm of $z \in \mathbb{C}^n$. For a Hermitian polynomial $M(z,\bar{z}) \in \mathbb{C}[z_1,\ldots,z_n,\bar{z}_1,\ldots,\bar{z}_n]$ (real-valued), the problem of whether $M(z,\bar{z})\|z\|^2$ can be expressed as a sum of norm squares of polynomials in $z$ (the \textit{sums of squares (SOS) problem}) has drawn significant attention, motivated by Hermitian analogues of Hilbert's 17th problem and proper holomorphic map studies (see \cite{Da,CD1,CD2}).  

Recently, Ebenfelt \cite{Eb} posed a related problem on the \textit{rank} of $M(z,\bar{z})\|z\|^2$: if expressible as a sum of norm squares, its rank is the dimension of the spanning vector space (equivalently, the minimal number of such polynomials). Inspired by Huang's Lemma \cite{Hu} and the Huang-Ji-Yin \textit{Gap Conjecture} \cite{HJY}, Ebenfelt proposed the \textit{SOS Conjecture} on possible ranks, showing the Gap Conjecture follows from it via a CR Gauss equation.  

Evidence for the SOS Conjecture includes monomial cases in $\mathbb{C}^3$ in \cite{BG}, $M(z,\bar{z})$ as SOS in \cite{GK}, and some cases in \cite{BGS} and \cite{WYZ}. 

In \cite{GN2}, the authors generalized the SOS conjecture to arbitrary signature and proved that the ranks of the Hermitian polynomials of the form  $M(z,\bar{z})\|z\|_{s,t}^2$ can only lie in certain disjoint intervals on the real line.

Let $ s, t\in \{0, \ldots, n\} $ with $(s, t) \neq (0, 0) $ and $s + t  = n $. Consider a Hermitian form on $\mathbb{C}^n $ (which may be indefinite or degenerate) defined by  
\[
\langle z, w \rangle_{s,t} := z_1 \overline{w}_1 + \cdots + z_s \overline{w}_s - z_{s+1} \overline{w}_{s+1} - \cdots - z_{s+t} \overline{w}_{s+t}
\]  
for $z, w \in \mathbb{C}^n $. This form $\langle \cdot, \cdot \rangle_{s,t} $ has signature $(s, t) $. We denote $\langle z, z \rangle_{s,t} $ by $\|z\|^2_{s,t} $. In general, we use $\langle \cdot, \cdot \rangle_{\star} $ to denote a Hermitian form on $\mathbb{C}^n $.

Let $M(z,\bar z)\in\mathbb C[z_1,\ldots,z_{n},\bar z_1,\ldots,\bar z_{n}]$ be a non-zero Hermitian \textit{bihomogeneous} polynomial with degree $(d,d)$. Then
there exists a Hermitian matrix $\mathcal M$ such that
	$$M(z,\bar{z}) = Z_d^H \mathcal M Z_d,$$
	where $Z_d$ is the vector of monomials on $\mathbb C^n$ of degree $d$. The rank $r$ of $\mathcal M$ is known as the rank of $M(z,\bar{z})$. Moreover, there exist linearly independent holomorphic functions $m_1(z), \ldots, m_{p+q}(z)$ such that
	$$M(z,\bar{z}) = |m_1(z)|^2 + \cdots + |m_p(z)|^2 - |m_{p+1}(z)|^2 + \cdots - |m_{p+q}(z)|^2.$$
	
	While $m_1(z), \ldots, m_{p+q}(z)$ are not unique, the signature $(p,q)$ and the rank $r(\mathcal M) = p+q$ of $M(z,\bar{z})$ are unique. 
	
	Now we consider the polynomial 
	$$F(z,\bar z)=M(z,\bar{z})\|z\|^2_{s,t}.$$
	Then $F(z,\bar z)$ is a homogeneous Hermitian polynomial of bidegree $(d+1,d+1)$ on $\mathbb{C}^n$ and there exists a matrix $\mathcal F$ such that 
	$$F(z,\bar z)=Z_{d+1}^H\mathcal F Z_{d+1}.$$
		If the matrix $\mathcal F$ has $s'$ positive eigenvalues, $t'$ negative eigenvalues, then there exist $s'+t'$ polynomials $f_1(z)$,$\cdots$,$f_{s'+t'}(z)$ determined by the eigenvectors of $\mathcal F$ such that 
	$$M(z,\bar z)\|z\|_{s,t}=|f_1|^2+\cdots+|f_{s}|^2-|f_{s'+1}|^2-\cdots-|f_{s'+t'}|^2.$$	
	Then $f=[f_1,\cdots,f_{s'+t'}]$ can define an proper mapping between generalized balls.  Consequently, the results concerning the aforementioned equations can be applied to study the classification of proper mappings between generalized balls. This idea can also be found in \cite{DL}.
	
	 Now we give the estimate of the rank $R$ of $F$ from the rank $r$ of $M(z,\bar z)$.

\begin{theorem}\label{rank}
	Let  $M(z,\bar{z})$ be a bihomogeneous Hermitian polynomial with rank $r$ on $\mathbb{C}^n$, 
	and let $R$ be the rank of $M(z,\bar{z})\|z\|^2_{\star}$
 Then 
	$$2r_{(n-1)}|^{1}_0-rn\le R \le rn.$$
\end{theorem}
\begin{proof}
	Assume that $\|z\|^2_{\star}=|z_1|^2+\cdots+|z_s|^2-|z_{s+1}|^2-\cdots-|z_{s+t}|^2$ for some non-negative integers with $s+t=n$.
Then $$F(z,\bar z)=\sum_{j = 1}^s\sum_{i=1}^{p}|m_i(z)z_j|^2-\sum_{j = s+1}^{s+t}\sum_{i=1}^{p}|m_i(z)z_{j}|^2+\sum_{j = s+1}^{s+t}\sum_{i=1}^{p}|m_{i}(z)z_{j}|^2-\sum_{j = 1}^s\sum_{i=p+1}^{p+q}|m_i(z)z_j|^2$$

	Denote the vector $(m_1z_1,\cdots,m_{r}z_{n})^t=Y$. Then 
$F(z,\bar z)=Y^H\Pi Y$, where $\Pi$ is an $rn\times rn$ invertible matrix whose diagonal entries are all $\pm 1$.

	Let $I$ be the ideal of $\mathbb C[z_1,\cdots,z_n]$ generated by $m_1(z)$, $\cdots$,$m_{p+q}(z)$ and $I_{d+1}$ be $(d+1)$-graded ideal of $I$. 
Then for any $1\le i\le{p+q}$, $1\le j\le n$, $m_i(z)z_j\in I_{d+1}$ and $I_{d+1}$ is generated by $m_i(z)z_j$ as a vector space. Assume that $\dim I_{d+1}=k$. Then there exists $k$ linear independent polynomials $g_1$,$\cdots$,$g_k$ such that $I_{d+1}$ is generated by $g_1$,$\cdots$,$g_k$.  Denote $g=(g_1(z),\cdots,g_k(z))^t$. 

There exists an  $rn\times k$ matrix $A$ such that  $(m_1z_1,\cdots,m_{r}z_n)^t=Y=A(g_1(z),\cdots,g_k(z))^t$ with $\text{rank}(A)=k$. Then
$$F(z,\bar z)=Y^H\Pi Y=g^HA^H\Pi Ag.$$
 Hence,   $$ \text{rank}F =\text{rank}(A^H\Pi A)\ge \text{rank}(A^H)+\text{rank}(\Pi A)-rn=2\text{rank}(A^H)-rn=2k-rn$$

By Theorem \ref{M2}, $k=\dim(I_{d+1})\ge r_{(n-1)}|^{1}_0$.

On the other hand, it is clear that $\text{rank}A\le rn$.

Hence $2r_{(n-1)}|^{+1}_0-rn\le  R=\text{rank}(A^H\Pi A)\le rn $.
\end{proof}

As an application, we get the following result.

\begin{corollary}\label{general thm}
Let $M(z,\bar z)\in\mathbb C[z_1,\ldots,z_n,\bar z_1,\ldots,\bar z_n]$ be a bihomogeneous Hermitian polynomial  with $\text{rank}(M)=r\leq n-1$,
and let $R$ denote the rank of $M(z,\bar z)\|z\|^2_{\star}$. Then the following rank bounds hold
$$
rn-r(r-1)\leq R\leq rn.
$$
\end{corollary}
\begin{proof}
When $r=r(M)\le n-1$,
$r_{(n-1)}=\binom{n-1}{n-1}+ \binom{n-2}{n-2}+\cdots+\binom{n-r}{n-r}$. Hence
$$r_{(n-1)}|^{1}_0=\binom{n}{n-1}+ \binom{n-1}{n-2}+\cdots+\binom{n-r+1}{n-r}=rn-\frac{r(r-1)}{2}.$$
By Theorem \ref{rank}, 
$$rn-r(r-1)=2r_{(n-1)}|^{1}_0-rn\le R \le rn.$$

\end{proof}

\section{Application to Sum of norm Squared}

Let  $M(z,\bar{z})$  be a bihomogeneous Hermitian polynomial of bidegree  $(d,d)$  with signature  $(p,q)$ . In this section, we establish results concerning the relationship between the rank of  $M(z,\bar{z})\|z\|^2$  and the signature  $(p,q)$  of $M$, focusing on the case where  $M(z,\bar{z})\|z\|^2$  admits a representation as a sum of squared norms.

Assume  $M(z,\bar{z}) = |m_1|^2 + \cdots + |m_p|^2 - |m_{p+1}|^2 - \cdots - |m_{p+q}|^2$  for some holomorphic polynomials  $m_1, \cdots, m_{p+q} \in \mathbb{C}[z_1, \dots, z_n]$ .We denote by  $I_{m^+}$  the ideal of  $\mathbb{C}[z_1, \dots, z_n]$  generated by  $m_1, \cdots, m_p$ , by  $I_{m^-}$  the ideal generated by  $m_{p+1}, \cdots, m_{p+q}$ , and by  $I_{m}$  the ideal generated by  $m_1, \cdots, m_{p+q}$.

To characterize the ideal-theoretic relationships underlying the factorization of such Hermitian polynomials, we appeal to a fundamental result from  $[\text{Da}]$ . This lemma elucidates the interaction between the ideals associated with the components of the factorization, serving as a key ingredient for our subsequent reasoning:

\begin{lemma}[\cite{Da}]\label{contain}Let  $M(z,\bar{z})$  be a bihomogeneous Hermitian polynomial of bidegree  $(d,d)$  on  $\mathbb{C}^n$ . If  $M(z,\bar{z})\|z\|^{2l} = \|h\|^2$  for some  $l \in \mathbb{N}$  and some holomorphic polynomials $h_1$, $\dots$, $h_k$ , then  $I_{m^-}(d+l) \subset I_{m^+}(d+l)$  and  $I_h(d+l) \subset I_{m^+}(d+l)$. Consequently,  $I_{m}(d+l) = I_{m^+}(d+l)$.\end{lemma}
	
	Let $C_{m}^k=\binom{m}{k}$ for any integers $m\ge k\ge 1$.
	
	Theorem 2 in \cite{GK} characterizes an inequality for  $ p $  (from the signature  $(p, q)$ ) of a bihomogeneous polynomial  $ M(z, \overline{z}) $  of bidegree  $(d, d)$  and rank  $ r = p + q $ , when  $ M(z, \overline{z})\|z\|^{2l} $  is a sum of squared norm. It provides a lower bound for  $ p $  that depends on the bidegree  $ d $  of  $ M(z, \overline{z}) $ .  
	
	The result that follows offers a lower bound for  $ p $  with a more explicit formulation and enhanced tractability, and it does not depend on the bidegree of  $ M(z, \overline{z}) $ .

\begin{theorem}
	Let $M(z,\bar z)$ be a bihomogeneous polynomial on $\mathbb C^n$  with signature $(p,q) $ and rank $r=p+q$ on $\mathbb C^n$. If $M(z,\bar z)\|z\|^{2l}$ is sum of squared norms,  then 
	$$p\ge \frac{r_{(n-1)}\big|_{0}^{l}}{C_{n-1+l}^l}.$$
In particular, when $l=1$, 
	$p\ge \frac{r_{(n-1)}\big|_{0}^{1}}{n}$. 
\end{theorem}
\begin{proof}
By Lemma \ref{contain}, $H_{I_{m}}(d+l)=H_{I_{m^{+}}}(d+l)$. From Theorem \ref{M2}, then 
 $$r_{(n-1)}\big|_{0}^{l}\le H_{I_{m}}(d+l)=H_{I_{m^{+}}}(d+l)\le pC_{n-1+l}^l.$$

Hence, we get the result.
\end{proof}

Now, we will introduce some results about the operator $m_{(n)}|_t^{s}$.
\begin{proposition}\label{b3}
For any positive integers $m$, $n$ and $l$, 
$$m_{(n)}|_0^{l}-m\ge m_{(n)}|_{-1}^{l-1}.$$ 
\end{proposition}
\begin{proof}
Assume that $m_{(n)}=\binom{a_n}{n}+\cdots+\binom{a_\delta}{\delta}$, where  $\delta \ge 1$, $a_k \ge k$ for all $k \in \{\delta, \dots, n\}$, and $a_k > a_{k-1}$.

Since  $\binom{a_k+l}{k}-\binom{a_k}{k}=\binom{a_{k}+l-1}{k-1}+\binom{a_{k}+l-2}{k-1}+\cdots+\binom{a_{k}}{k-1}\ge \binom{a_{k}+l-1}{k-1}$ for any $k$, then $m_{(n)}|_0^{l}-m\ge m_{(n)}|_{-1}^{l-1}.$
\end{proof}
By similar computation, we get the following result.
\begin{proposition}\label{b4}
	For any positive integers $m$, $n$ and $k$, 
	$$m_{(n)}|^{k}_{-1}\ge (m-1)_{(n)}|^{k}_{-1}.$$ 
\end{proposition}
Inspired by Corollary 1 in \cite{GK}, the following corollary states the condition on  the signature $(p, q)$ for $M(z,\bar z)\|z\|^{2l}$ to be a sum of norm squares. The inequality in the corollary generalizes both Corollary 1 in \cite{GK} and Main Theorem 1.1 in \cite{HL}. Notably, the proof avoids combinatorial graph analysis (used in \cite{HL}) and relies directly on ideal theory, making it more elementary. We further establish a rank bound for such polynomials:

\begin{corollary}\label{Cor}
	Let $M(z,\bar z)$ be a bihomogeneous polynomial of bidegree $(d,d)$ with the signature pair $(p,q)$ and suppose $M(z,\bar z)\|z\|^{2l}$ is a sum of norm squares. Then 
	$$q\le pC_{n-1+l}^l-p-p_{(n-1)}\big|^{l-1}_{-l}.$$
	In particular, when $l=1$,
	$$q\le p(n-1)-p_{(n-1)}\big|^{-1}_{-l}.$$
\end{corollary}
\begin{proof}
By Theorem \ref{M2} and Lemma \ref{contain}, we have $(p+q)_{(n-1)}\big|_{0}^{l}\le H_{I_{m^{+}\oplus m^{-}}}(d+l)=H_{I_{m^{+}}}(d+l)\le pC_{n-1+l}^l$. 

From Proposition \ref{b3} and \ref{b4},  $ (p+q)_{(n-1)}\big|_{0}^{l}\ge (p+q)_{(n-1)}\big|_{-1}^{l-1}+(p+q) \ge p_{(n-1)}\big|_{-1}^{l-1}+(p+q)$. Therefore, 
$$q\le pC_{n-1+l}^l-p-p_{(n-1)}\big|_{-1}^{l-1}.$$

\end{proof}

In addition, we can get a new inequality to estimate the rank of $R$ which partially answer the SOS conjecture.

\begin{theorem}\label{SOS1}
Let $M(z,\bar z)$ be a bihomogeneous polynomial on $\mathbb C^n$ of bidegree $(d,d)$ with the signature pair $(p,q)$. Suppose $M(z,\bar z)\|z\|^{2l}=\|h\|^{2}$ and $R$ be the rank of $\|h\|^2$. Then 
$$(p+q)_{(n-1)}|_0^l-qC_{n-1+l}^l \le R \le pC_{n-1+l}^l.$$
\end{theorem}
\begin{proof}
	Assume $M(z,\bar z)=\sum_{i=1}^{p}\|f_i\|^2-\sum_{j=1}^{q}\|g_j\|^2$. By Lemma \ref{contain}, we get 
	 $$(p+q)_{(n-1)}|_0^l-qn\le H_{I_{f\oplus g}}(d+l)- H_{I_g}(d+l)\le R=H_{I_{h}}(d+l) \le H_{I_{f}}(d+l)  \le pn.$$
\end{proof}

Theorem \ref{SOS1} is particularly effective when $q$ (the negative signature component) is small. For example, if $q = 0$ (i.e., $M$ itself is a sum of squared absolute values), the lower bound simplifies to  $p_{(n-1)}\big|_0^l \le R$, which recovers Proposition 3 in \cite{GK}-a key result supporting the SOS conjecture for positive semi-definite Hermitian polynomials.

\section{Appendix}\label{App}

\begin{lemma}
	Suppose $m,k,s\geq 1$ and $A,B\geq 0$. If $A+B=\binom{m+k}{k}$, then $$A_{(m)}|_0^{s}+B_{(k)}|^{s}_{s}=\binom{m+k+s}{k+s}$$
	Here, we make the convention that $0_{(m)}|_0^{s}=0_{(k)}|^{s}_{s}=0$.
\end{lemma}

\begin{proof}
	By induction, we only need to prove the case $s=1$.
	
	At first, we want to prove the case $A=0$ or $B=0$.
	
	when $A=0$, then $B_{(k)}=\binom{m+k}{k}$. It is clear that $A_{(m)}|_0^{+1}+B_{(k)}|^{+1}_{+1}=\binom{m+k+1}{k+1}$.
	when $B=0$,then $A_{(m)}=\binom{m+k}{m}$. It is easy to get $A_{(m)}|_{0}^s+B_{(k)}|_{s}^s=\binom{m+k+1}{m}=\binom{m+k+1}{k+1}$. 
	
	Now we only need to prove the case $A>0$ and $B>0$.We will prove by induction and first show that the lemma is true for $m=1$ or $k=1$.

	Suppose $k=1$ and $A+B=\binom{m+1}{1}=m+1$.
	
	If $1\leq A\leq m$, then
	$
	A_{(m)}=\binom{m}{m}+\binom{m-1}{m-1}+\cdots+\binom{m-A+1}{m-A+1}
	$ and $B_{(1)}=\binom{m-A+1}{1}$.
	Hence, \begin{eqnarray*}
		A_{(m)}|_0^{1}+B_{(k)}|^{1}_{1}&=&\binom{m+1}{m}+\binom{m}{m-1}+\cdots+\binom{m-A+2}{m-A+1}+\binom{m-A+2}{2}\\
		&=&\binom{m+2}{m}=\binom{m+2}{2}.
	\end{eqnarray*}

	Suppose $m=1$ and $A+B=\binom{1+k}{k}$. Then $1\leq B\leq k$ and hence $B_{(k)}=\binom{k}{k}+\binom{k-1}{k-1}+\cdots+\binom{k-B+1}{k-B+1}$ and $A_{(1)}=\binom{k-B+1}{1}$. Hence, 
	\begin{eqnarray*}
			A_{(1)}|_0^{1}+B_{(k)}|^{1}_{1}&=&\binom{k-B+2}{1}+\binom{k+1}{k+1}+\binom{k}{k}+\cdots+\binom{k-B+2}{k-B+2}\\
			&=&\binom{k+2}{2}. 
		\end{eqnarray*}

	Suppose $A+B=\binom{m+k}{k}=\binom{m+k}{m}=\binom{m+k-1}{m}+\binom{m+k-1}{k}$. Then, 
	either $A\geq\binom{m+k-1}{m}$ or $B\geq\binom{m+k-1}{k}$. Otherwise, $A+B\leq \binom{m+k-1}{m}-1+\binom{m+k-1}{k}-1$.

	If  $A\geq \binom{m+k-1}{m}$, then 
	$A_{(m)}=\binom{m+k-1}{m}+\binom{a_{m-1}}{m-1}+\cdots+\binom{a_\delta}{\delta}$. Since $(A-\binom{m+k-1}{m})+B=\binom{m-1+k}{k}$, by the induction hypothesis, we get
	{\footnotesize
		\begin{eqnarray*}
			\left(A-\binom{m+k-1}{m}\right)_{(m-1)}|_{0}^{1}+B_{(k)}|_{1}^{1}&=&\binom{m+k}{k+1}\\
			\Rightarrow\,\,\,\,\, \binom{a_{m-1}+1}{m}+\cdots+\binom{a_\delta+1}{\delta}+B_{(k)}|_{1}^{1}&=&\binom{m+k}{k+1}\\
			\Rightarrow\,\,\,\,\, \binom{m+k}{m}+\binom{a_{m-1}-1}{m-2}+\cdots+\binom{a_\delta-1}{\delta-1}+B_{(k)}|_{+1}^{+1}&=&\binom{m+k}{k+1}+\binom{m+k}{m}\\
			\Rightarrow\,\,\,\,\, A_{(m)}|_0^{1}+B_{(k)}|^{1}_{1}&=&\binom{m+k+1}{k+1}
		\end{eqnarray*}
	}
	
	If  $B\geq \binom{m+k-1}{k}$, then the $k$-th Macaulay's representation of $B$ is of the form 
	$B_{(k)}=\binom{m+k-1}{k}+\binom{b_{k-1}}{k-1}+\cdots+\binom{b_\epsilon}{\epsilon}$. Since $A+(B-\binom{m+k-1}{k})=\binom{m+k-1}{k-1}$, by the induction hypothesis, we get
	{\footnotesize
		\begin{eqnarray*}
			A_{(m)}|_0^{1}+\left(B-\binom{m+k-1}{k}\right)_{(k-1)}|^{1}_{1}&=&\binom{m+k}{k}\\
			\Rightarrow\,\,\,\,\, 	A_{(m)}|_0^{1}+\binom{b_{k-1}+1}{k}+\cdots+\binom{b_\epsilon+1}{\epsilon+1}&=&\binom{m+k}{k}\\
			\Rightarrow\,\,\,\,\, 	A_{(m)}|_0^{1}+\binom{m+k}{k+1}+\binom{b_{k-1}+1}{k}+\cdots+\binom{b_\epsilon+1}{\epsilon+1}&=&\binom{m+k}{k+1}+\binom{m+k}{k}\\
			\Rightarrow\,\,\,\,\, A_{(m)}|_0^{1}+B_{(k)}|^{1}_{1}&=&\binom{m+k+1}{k+1}
		\end{eqnarray*}
	}
\end{proof}

\end{document}